\documentclass[11pt]{article}
\usepackage{mathrsfs}
\usepackage{mathtools}
\usepackage{amsfonts}
\usepackage{amssymb}
\usepackage{graphicx}
\usepackage{amsmath,amsthm}

\def\sqr#1#2{{\vcenter{\vbox{\hrule height.#2pt
              \hbox{\vrule width.#2pt height#1pt \kern#1pt \vrule
width.#2pt}
              \hrule height.#2pt}}}}
\def\3n{\negthinspace \negthinspace \negthinspace }
\def\2n{\negthinspace \negthinspace }
\def\1n{\negthinspace }

\def\ms{\medskip}

\def\q{\quad}

\def\({\Big (}
\def\){\Big )}
\def\[{\Big[}
\def\]{\Big]}

\def\be{\begin{equation}}
\def\bel{\begin{equation}\label}
\def\ee{\end{equation}}
\def\bea{\begin{eqnarray}}
\def\eea{\end{eqnarray}}
\def\bt{\begin{theorem}}
\def\et{\end{theorem}}
\def\bc{\begin{corollary}}
\def\ec{\end{corollary}}
\def\bl{\begin{lemma}}
\def\el{\end{lemma}}
\def\bp{\begin{proposition}}
\def\ep{\end{proposition}}
\def\br{\begin{remark}}
\def\er{\end{remark}}
\def\ba{\begin{array}}
\def\ea{\end{array}}
\def\bd{\begin{definition}}
\def\ed{\end{definition}}

\newtheorem{lemma}{Lemma}[section]
\newtheorem{remark}{Remark}[section]

\newtheorem{theorem}{Theorem}[section]
\newtheorem{corollary}{Corollary}[section]

\newtheorem{definition}{Definition}[section]
\newtheorem{proposition}{Proposition}[section]

\makeatletter
   
   \@addtoreset{equation}{section}
\makeatother
\allowdisplaybreaks
\begin{document}

\title{\bf p-Laplacian wave equations in non-cylindrical domains \thanks{This work is
supported by the NSF of China under grants 11371084, 11471070 and 11171060.}}

\author{Lingyang Liu\thanks{School of Mathematics and Statistics, Northeast Normal
University, Changchun 130024, China. E-mail address:
liuly938@nenu.edu.cn \ms } \q
 and \q Hang Gao\thanks{School of Mathematics and Statistics, Northeast Normal
University, Changchun 130024, China. E-mail address:
hanggao2013@126.com
 \ms }}

\date{}

\maketitle

\begin{abstract}
This paper is devoted to studying the stability of p-Laplacian wave equations with strong damping in non-cylindrical domains. The method of proof based on some estimates for time-varying coefficients rising from moving boundary and a modified Kormonik inequality. Meanwhile, by selecting appropriate auxiliary functions, finally we obtain the polynomial stability ($p>2$) and exponential stability ($p=2$) for such systems in some unbounded development domains.

\end{abstract}


\section{Introduction and preliminaries}
In this paper, we consider the following quasi-linear hyperbolic system with strongly damped terms
\begin{equation}\label{4.1}
\left\{\begin{array}{ll}
u_{tt}-\operatorname{div}(|\nabla u|^{p-2}\nabla u)-\Delta u_t=0& \mbox{in}\ \widehat{Q}, \\[3mm]
u=0 &\mbox{on}\ \widehat{\Sigma}, \\[3mm]
u(x,0)=u_0(x), \ u_t(x,0)=u_1(x)&\mbox{in}\ \Omega_0,
\end{array}\right.
\end{equation}
where $\widehat{Q}$ is a non-cylindrical domain in $\mathbb R^2$, $\widehat{\Sigma}$ is the lateral boundary of $\widehat{Q}$, $\Omega_t=\widehat{Q}\cap\left(\mathbb R\times\{t\}\right)$, $u$ is the state variable, $(u_0,u_1)$ is any given initial value and $p\geq2.$

We assume that $\widehat{Q}$ is increasing with time, that is,
\begin{equation}\label{a1}
\Omega_t\subset \Omega_s,\quad t<s
\end{equation}

Moreover, we assume the following regular property
\begin{eqnarray}\label{a2}
\mbox{if} \quad u\in W^{1,p}_0(\Omega) \quad \mbox{and}  \quad  u\mid_{\Omega-\Omega_t}=0, \quad \mbox{then} \quad u\mid_{\Omega_t}\in H^1_0(\Omega_t).
\end{eqnarray}

\begin{lemma}[\cite{FRS}]\label{lem4.1}
Suppose that $\eqref{a1}-\eqref{a2}$ hold. Then given $u_0\in W^{1,p}_0(\Omega_0)$,
$u_1\in L^2(\Omega_0)$, there exists a function $u$: $\widehat{Q}\rightarrow \mathbb R$ such that
\begin{eqnarray*}
u\in L^\infty\big(0,T;W_0^{1,p}(\Omega_t)\big),
\end{eqnarray*}
\begin{equation*}
u_t\in L^\infty\big(0,T;L^2(\Omega_t)\big)\bigcap L^2\big(0,T;H^1_0(\Omega_t)\big),
\end{equation*}
\begin{equation*}
u(0)=u_0\quad \mbox{and} \quad u_t(0)=u_1 \quad a.e.\ \mbox{in}\ \Omega_0,
\end{equation*}
\begin{equation*}
u_{tt}-\operatorname{div}(|\nabla u|^{p-2}\nabla u)-\Delta u_t=0 \quad \mbox{in}\ L^2\big(0,T;W^{-1,p'}(\Omega_t)\big).
\end{equation*}
\end{lemma}

Define the energy of system \eqref{4.1} by
\begin{eqnarray*}
E(t)=\frac{1}{2}||u_t||^2_2+\frac{1}{p}||\nabla u||^p_p.
\end{eqnarray*}

\begin{proposition}
$E(t)$ is non-increasing with respect to the time variable.
\end{proposition}
\begin{proof}It is easy to see that
$$
\displaystyle\frac{d}{dt}E(t)=\frac{1}{2}\frac{d}{dt}||u_t||^2_2+\frac{1}{p}\frac{d}{dt}||\nabla u||^p_p.
$$
Since
\begin{eqnarray*}
\frac{d}{dt}\int_{\Omega_t}u(x,t)dx=\int_{\Omega_t}\frac{\partial}{\partial t}u(x,t)dx-\int_{\Gamma_t}u(x,t)v_td\sigma,
\end{eqnarray*}
where $v=(v_x,v_t)^\top$ is the external unit normal on $\widehat{\Sigma}$ relative to $\widehat{Q}$,

we have
\begin{eqnarray*}\label{4.2}
\displaystyle\frac{1}{2}\frac{d}{dt}||u_t||^2_2\!\!\!\!\!\!\!\!&&=\displaystyle\int_{\Omega_t}u_tu_{tt}dx-\frac{1}{2}\int_{\Gamma_t}u^2_tv_td\sigma\\[2mm]
&&\displaystyle=\int_{\Omega_t}u_t\left[\operatorname{div}\left(|\nabla u|^{p-2}\nabla u\right)+\Delta u_t\right]dx-\frac{1}{2}\int_{\Gamma_t}u^2_tv_td\sigma\\[2mm]
&&\displaystyle=\int_{\Gamma_t}u_t|\nabla u|^{p-2}\nabla u v_xd\sigma-\int_{\Omega_t}|\nabla u|^{p-2}\nabla u \nabla u_tdx\\[2mm]
&&\displaystyle\quad+\int_{\Gamma_t}u_t\nabla u_tv_xd\sigma-\int_{\Omega_t}|\nabla u_t|^2-\frac{1}{2}\int_{\Gamma_t}u_t^2v_td\sigma,
\end{eqnarray*}
and
$$\frac{1}{p}\frac{d}{dt}||\nabla u||^p_p=\int_{\Omega_t}|\nabla u|^{p-2}\nabla u \nabla u_tdx-\frac{1}{p}\int_{\Gamma_t}|\nabla u|^pv_td\sigma.$$
Hence
\begin{eqnarray*}\label{4.3}
\frac{d}{dt}E(t)=\int_{\Gamma_t}\left(u_t|\nabla u|^{p-2}\nabla u v_x-\frac{1}{p}|\nabla u|^pv_t-\frac{1}{2}u^2_tv_t+u_t\nabla u_tv_x\right)d\sigma-\int_{\Omega_t}|\nabla u_t|^2dx.
\end{eqnarray*}
On the other side, from $u=0$ on $\Sigma,$ one gets $\nabla_{x,t}u=\frac{\partial u}{\partial v} v$, that is, $\nabla u=\frac{\partial u}{\partial v} v_x$, $u_t=\frac{\partial u}{\partial v} v_t$. Besides, noticing that $u_t=0$ on $\Sigma,$ we arrive at
\begin{eqnarray*}
\begin{array}{rl}
\displaystyle\frac{d}{dt}E(t)=-\int_{\Omega_t}|\nabla u_t|^2dx.
\end{array}
\end{eqnarray*}
\end{proof}

\begin{lemma}[\cite{BM}]\label{lem4.2}
Let $E:$ $\mathbb R^+\rightarrow \mathbb R^+$ be a non-increasing function and $\phi:$ $\mathbb R^+\rightarrow \mathbb R^+$ be a strictly increasing $C^1$ function such that
\begin{eqnarray*}\label{4.4}
\phi(0)=0 \quad \mbox{and} \quad \phi(t)\rightarrow+\infty \quad \mbox{as}\ t\rightarrow+\infty.
\end{eqnarray*}
If there exist constants $q\geq0$ and $A>0$ such that
\begin{eqnarray*}
\begin{array}{rl}
&\displaystyle\int^{+\infty}_S E^{q+1}(t)\phi'(t)dt\leq\frac{1}{A}E^q(0)E(S),\quad 0\leq S<+\infty,\\[1mm]
\end{array}
\end{eqnarray*}
then we have
\begin{eqnarray*}
\begin{array}{rl}
&\displaystyle E(t)\leq E(0)\left(\frac{1+q}{1+qA\phi(t)}\right)^{\frac{1}{q}}\quad \forall t\geq0,\quad \mbox{for}\ q>0;
\end{array}
\end{eqnarray*}
and
\begin{equation*}
E(t)\leq E(0)e^{1-A\phi(t)}\quad \forall t\geq0, \quad \mbox{for}\ q=0.
\end{equation*}
\end{lemma}

\begin{proposition}\label{pro4.8}
For any $t\geq0,$ it holds that
\begin{eqnarray*}\label{4.8}
\begin{array}{rl}
\displaystyle||\nabla u(t)||^2_2\leq|\Omega_t|^{1-\frac{2}{p}}||\nabla u(t)||^2_p,
\end{array}
\end{eqnarray*}
\begin{equation*}\label{4.9}
\displaystyle||u(t)||^2_2\leq|\Omega_t|^{3-\frac{2}{p}}||\nabla u(t)||^2_p.
\end{equation*}
\end{proposition}

\begin{proof}
By H\"{o}lder's inequality, one has
\begin{eqnarray*}\label{4.5}
&&\int_{\Omega_t}u^2(t)dx=\int_{\Omega_t}\left[\int^x_0u_y(t)dy\right]^2dx\leq |\Omega_t|^2\int_{\Omega_t}u^2_x(t)dx,
\end{eqnarray*}
that is
\begin{eqnarray*}\label{4.6}
&\displaystyle ||u(t)||^2_2\leq |\Omega_t|^2||\nabla u(t)||^2_2.
\end{eqnarray*}
Moreover
\begin{eqnarray*}\label{4.7}
\displaystyle||\nabla u(t)||_2=\left[\int_{\Omega_t}|\nabla u(t)|^2dx\right]^{\frac{1}{2}}
\leq|\Omega_t|^{\frac{1}{2}-\frac{1}{p}} \left[\int_{\Omega_t}|\nabla u(t)|^pdx\right]^{\frac{1}{p}}=|\Omega_t|^{\frac{1}{2}-\frac{1}{p}}||\nabla u||_p.
\end{eqnarray*}
Thus
\begin{eqnarray*}
\begin{array}{rl}
\displaystyle||\nabla u(t)||^2_2\leq|\Omega_t|^{1-\frac{2}{p}}||\nabla u(t)||^2_p,
\end{array}
\end{eqnarray*}
\begin{equation*}
\displaystyle||u(t)||^2_2\leq|\Omega_t|^{3-\frac{2}{p}}||\nabla u(t)||^2_p.
\end{equation*}
\end{proof}

\section{Main results and proofs}
The main result of this paper is stated as follows.
\begin{theorem}
Let $k>0,$ $0<\gamma<1$ and $m>0.$ For domain $\widehat{Q}$ with $|\Omega_t|=(1+kt)^{\frac{1-\gamma}{m}}$, we have

\vspace{1mm}
$(1)$ when $p>2,$ and $m\geq \max\left\{2, \frac{(1-\alpha)p}{\alpha p-1}, \frac{\left(\frac{1}{2}+\alpha\right)p-1}{\left(\frac{1}{2}-\alpha\right)p+1}\right\}$ for
$\frac{1}{p}<\alpha<\frac{1}{2}+\frac{1}{p}$, the energy of system \eqref{4.1} satisfies
\begin{equation*}
E(t)\leq \left[\frac{C(1+\beta)\left(E^\beta(0)+1\right)}{\beta}\right]^{\frac{1}{\beta}}\left[(1+kt)^\gamma-1\right]^{-\frac{1}{\beta}}, \quad \beta=\frac{p-2}{p};
\end{equation*}

$(2)$ when $p=2,$ and $m\geq 2$, the energy of system \eqref{4.1} satisfies
\begin{equation*}
E(t)\leq E(0)e^{1-\frac{1}{C}\left[(1+kt)^\gamma-1\right]}.
\end{equation*}
Here and henceforth, $C$ is some positive constant.
\end{theorem}

\begin{proof}
Multiply the first equation in \eqref{4.1} by $E^\beta(t)\phi'(t)u$, and integrating it over $\Omega_t\times(S,T)$, we get
\begin{eqnarray*}\label{4.10}
0\!\!\!\!\!\!\!\!&&=\int^T_{S}\int_{\Omega_t} E^\beta(t)\phi'(t)u\left(u_{tt}-\operatorname{div}(|\nabla u|^{p-2}\nabla u)-\Delta u_t\right)dxdt\\[2mm]
&&=\int^T_{S}\int_{\Omega_t} E^\beta(t)\phi'(t)uu_{tt}dxdt-\int^T_{S}\int_{\Omega_t} E^\beta(t)\phi'(t)u\operatorname{div}(|\nabla u|^{p-2}\nabla u)dxdt\\[2mm]
&&\quad-\int^T_{S}\int_{\Omega_t} E^\beta(t)\phi'(t)u\Delta u_tdxdt\\[2mm]
&&\triangleq I_1+I_2+I_3
\end{eqnarray*}
Calculating $I_i$ $(i=1,2,3)$, respectively, we have
\begin{eqnarray*}\label{4.11}
&&I_1=\int_{\Omega_T}E^\beta(T)\phi'(T)u(T)u_t(T)dx-\int_{\Omega_S}E^\beta(S)\phi'(S)u(S)u_t(S)dx\\[2mm]
&&\qquad-\beta\int^T_{S}E^{\beta-1}(t)E'(t)\int_{\Omega_t}\phi'(t)uu_tdxdt-\int^T_{S}\int_{\Omega_t} E^\beta(t)\phi''(t)uu_tdxdt\\[2mm]
&&\qquad-\int^T_{S}\int_{\Omega_t} E^\beta(t)\phi'(t)u^2_tdxdt,
\end{eqnarray*}
\begin{eqnarray*}\label{4.12}
\begin{array}{rl}
&\displaystyle I_2=\int^T_{S}\int_{\Omega_t}E^\beta(t)\phi'(t)|\nabla u|^pdxdt,
\end{array}
\end{eqnarray*}
and
\begin{eqnarray*}\label{4.13}
\begin{array}{rl}
&\displaystyle I_3=\int^T_{S}\int_{\Omega_t}E^\beta(t)\phi'(t)\nabla u\nabla u_t dxdt.
\end{array}
\end{eqnarray*}
In view of
$$
I_1+I_2+I_3=0,
$$
one has
$$
I_2=-I_1-I_3.
$$
By the definition of $E(t)$, one gets
\begin{eqnarray}\label{4.14}
&&||\nabla u||^p_p=pE(t)-\frac{p}{2}||u_t||^2_2.
\end{eqnarray}
Applying \eqref{4.14} to $I_2$, we obtain
\begin{eqnarray*}\label{4.15}
&&\displaystyle\int^T_{S}E^\beta(t)\phi'(t)\left[pE(t)-\frac{p}{2}\int_{\Omega_t}u^2_tdx\right]dt=-I_1-I_3.
\end{eqnarray*}
Substituting $I_1$ and $I_3$ into the above equality, we get
\begin{eqnarray}\label{4.16}
&&\displaystyle p\int^T_{S}E^{\beta+1}(t)\phi'(t)dt\nonumber\\[2mm]
&&\displaystyle=\underbrace{\left(\frac{p}{2}+1\right)\int^T_{S}E^\beta(t)\phi'(t)\int_{\Omega_t}u^2_tdxdt}_{T_1}\nonumber\\[2mm]
&&\displaystyle\quad\underbrace{-\int_{\Omega_T}E^\beta(T)\phi'(T)u(T)u_t(T)dx+\int_{\Omega_S}E^\beta(S)\phi'(S)u(S)u_t(S)dx}_{T_2}\nonumber\\[2mm]
&&\displaystyle\quad\underbrace{+\beta\int^T_{S}E^{\beta-1}(t)E'(t)\int_{\Omega_t}\phi'(t)uu_tdxdt}_{T_3}\underbrace{+\int^T_{S}\int_{\Omega_t} E^\beta(t)\phi''(t)uu_tdxdt}_{T_4}\nonumber\\[2mm]
&&\displaystyle\quad\underbrace{-\int^T_{S}\int_{\Omega_t}E^\beta(t)\phi'(t)\nabla u\nabla u_t dxdt}_{T_5}.
\end{eqnarray}
We estimate $T_i$ $(i=1,2,3,4,5)$, respectively. First, we estimate the value of $T_1$. Using Poincar\'{e}'s inequality, we have
\begin{equation*}
\begin{array}{ll}
&\displaystyle T_1\leq \left(\frac{p}{2}+1\right)\int^T_{S}E^\beta(t)\phi'(t)|\Omega_t|^2\int_{\Omega_t}|\nabla u_t|^2dxdt.
\end{array}
\end{equation*}
Recall that $\phi'$ is a non-negative function in $\mathbb R^+$. If it satisfies
\begin{equation}\label{4.17}
\phi'(t)|\Omega_t|^2\leq C,
\end{equation}
then according to $E'(t)=-\int_{\Omega_t}|\nabla u_t|^2dx$, we deduce
\begin{equation*}
T_1\leq C\left(\frac{p}{2}+1\right)\int^T_{S}E^\beta(t)\left[-E'(t)\right]dt.
\end{equation*}
Sine $E$ is non-increasing, further
\begin{equation}\label{4.18}
T_1\leq C\left(\frac{p}{2}+1\right)\frac{1}{\beta+1}E^{\beta+1}(S).
\end{equation}

Secondly, we estimate the value of $T_2$. Using Cauchy-Schwarz's inequality and the conclusion in Proposition\ref{pro4.8}: $||u(t)||^2_2\leq|\Omega_t|^{3-\frac{2}{p}}||\nabla u(t)||^2_p$, we get
\begin{eqnarray*}\label{4.19}
T_2\!\!\!\!\!\!\!\!&&\leq\int_{\Omega_T}\left|E^\beta(T)\phi'(T)u(T)u_t(T)\right|dx+\int_{\Omega_S}\left|E^\beta(S)\phi'(S)u(S)u_t(S)\right|dx\\[2mm]
&&\leq E^\beta(T)\left(||u_t(T)||^2_2+\phi'^2(T)||u(T)||^2_2\right)+E^\beta(S)\left(||u_t(S)||^2_2+\phi'^2(S)||u(S)||^2_2\right)\\[2mm]
&&\leq E^\beta(T)\left(||u_t(T)||^2_2+\phi'^2(T)|\Omega_T|^{3-\frac{2}{p}}||\nabla u(T)||^2_p\right)\\[2mm]
&&\quad+E^\beta(S)\left(||u_t(S)||^2_2+\phi'^2(S)|\Omega_S|^{3-\frac{2}{p}}||\nabla u(S)||^2_p\right).
\end{eqnarray*}
If $\phi'(t)$ satisfies
\begin{equation}\label{4.20}
\phi'^2(t)|\Omega_t|^{3-\frac{2}{p}}\leq C,
\end{equation}
then by
\begin{equation*}
||u_t(t)||^2_2\leq2E(t)  \ \mbox{and} \ ||\nabla u(t)||^2_p\leq p^{\frac{2}{p}}E^{\frac{2}{p}}(t),
\end{equation*}
we obtain
\begin{equation}\label{4.21}
\begin{array}{ll}
T_2\leq C\left[E^{\beta+1}(S)+p^{\frac{2}{p}}E^{\beta+\frac{2}{p}}(S)\right].
\end{array}
\end{equation}

Thirdly, we estimate the value of $T_3$. By Cauchy-Schwarz's inequality, it follows that
\begin{eqnarray*}
T_3\!\!\!\!\!\!\!\!&&\leq\beta\int^T_{S}E^{\beta-1}(t)\left[-E'(t)\right]\int_{\Omega_t}\left|\phi'(t)uu_t\right|dxdt\\[2mm]
&&\leq\beta\int^T_{S}E^{\beta-1}(t)\left[-E'(t)\right]\left[\int_{\Omega_t}u^2_tdx+\int_{\Omega_t}\phi'^2(t)u^2dx\right]dt.
\end{eqnarray*}
Under the assumption \eqref{4.20}, similar to the estimation of $T_2$, we derive
\begin{equation}\label{4.22}
T_3\leq C\beta\left(\int^T_{S}E^\beta(t)\left[-E'(t)\right]dt- p^{\frac{2}{p}}\int^T_{S}E^{\beta+\frac{2}{p}-1}(t)E'(t)dt\right).
\end{equation}

Next, estimate the value of $T_4$. Using Young's inequality with $\varepsilon$, for $T_4$, we have the following estimation
\begin{eqnarray*}
T_4\!\!\!\!\!\!\!\!&&\leq\int^T_{S}E^\beta(t)\int_{\Omega_t}\left|\phi''(t)uu_t\right|dxdt\\[2mm]
&&\leq \varepsilon \int^T_{S}E^\beta(t)|\phi''(t)|^{\alpha p}\int_{\Omega_t}|u|^pdxdt+\varepsilon^{-\frac{q}{p}}\int^T_{S}E^\beta(t)|\phi''(t)|^{(1-\alpha )q}\int_{\Omega_t}|u_t|^qdxdt,
\end{eqnarray*}
where $\frac{1}{p}+\frac{1}{q}=1$, $0<\alpha<1$ and $\varepsilon>0$.

\vspace{1mm}
Noticing
\begin{equation*}
\int_{\Omega_t}|u|^pdx\leq |\Omega_t|^p\int_{\Omega_t}|\nabla u|^pdx,
\end{equation*}
and
\begin{equation*}
\int_{\Omega_t}|u_t|^qdx\leq|\Omega_t|^q\int_{\Omega_t}|\nabla u_t|^qdx\leq |\Omega_t|^{1+\frac{q}{2}}\left(\int_{\Omega_t}|\nabla u_t|^2dx\right)^{\frac{q}{2}}.
\end{equation*}
we get
\begin{eqnarray*}
T_4\!\!\!\!\!\!\!\!&&\leq\varepsilon p \int^T_{S}|\phi''(t)|^{\alpha p}|\Omega_t|^pE^{\beta+1}(t)dt+\varepsilon^{-\frac{q}{p}}\int^T_{S}E^\beta(t)|\phi''(t)|^{(1-\alpha )q}|\Omega_t|^{1+\frac{q}{2}}\left[-E'(t)\right]^{\frac{q}{2}}dt.
\end{eqnarray*}
Applying Young's inequality with $\eta$ to the second term at the right end of the above inequality, we deduce
\begin{eqnarray*}
&&\varepsilon^{-\frac{q}{p}}\int^T_{S}E^\beta(t)|\phi''(t)|^{(1-\alpha )q}|\Omega_t|^{1+\frac{q}{2}}\left[-E'(t)\right]^{\frac{q}{2}}dt\\[2mm]
&&\leq \varepsilon^{-\frac{q}{p}}\eta\int^T_{S}\left[E^\beta(t)|\phi''(t)|^{(1-\alpha )q}|\Omega_t|^{1+\frac{q}{2}}\right]^{\frac{1}{1-\frac{q}{2}}}dt+\varepsilon^{-\frac{q}{p}}\eta^{1-\frac{2}{q}}\int^T_{S}\left[-E'(t)\right]dt.
\end{eqnarray*}
If
\begin{equation}\label{4.24}
|\phi''(t)|^{\alpha p}|\Omega_t|^p\leq C\phi'(t),
\end{equation}
and
\begin{equation}\label{4.25}
\left[|\phi''(t)|^{(1-\alpha )q}|\Omega_t|^{1+\frac{q}{2}}\right]^{\frac{1}{1-\frac{q}{2}}}\leq C\phi'(t),
\end{equation}
hold, then it follows that
\begin{eqnarray}\label{4.26}
T_4\!\!\!\!\!\!\!\!&&\leq C\varepsilon p \int^T_{S}E^{\beta+1}(t)\phi'(t)dt+C\varepsilon^{-\frac{q}{p}}\eta\int^T_{S}E^{\frac{\beta}{1-\frac{q}{2}}}(t)\phi'(t)dt\nonumber\\[2mm]
&&\quad+\varepsilon^{-\frac{q}{p}}\eta^{1-\frac{2}{q}}\int^T_{S}\left[-E'(t)\right]dt.
\end{eqnarray}

Finally, we estimate $T_5$. By Young's inequality with $\delta$, one has
\begin{eqnarray*}
T_5\!\!\!\!\!\!\!\!&&\leq \int^T_{S}E^\beta(t)\phi'(t)\int_{\Omega_t}\left|\nabla u\nabla u_t\right|dxdt\\[2mm]
&&\leq\int^T_{S}E^\beta(t)\phi'(t)\left(\delta\int_{\Omega_t}\left|\nabla u\right|^pdx+\delta^{-\frac{q}{p}}\int_{\Omega_t}\left|\nabla u_t\right|^qdx\right)dt\\[2mm]
&&\leq \delta p\int^T_{S}E^{\beta+1}(t)\phi'(t)dt+\delta^{-\frac{q}{p}}\int^T_{S}E^\beta(t)\phi'(t)\int_{\Omega_t}\left|\nabla u_t\right|^qdxdt.
\end{eqnarray*}
Using H\"{o}lder's inequality, formula $E'(t)=-\int_{\Omega_t}|\nabla u_t|^2dx$ and Young's inequality with $\eta$, from the second term at the right end of the above inequality, one gets
\begin{eqnarray*}
&&\delta^{-\frac{q}{p}}\int^T_{S}E^\beta(t)\phi'(t)\int_{\Omega_t}\left|\nabla u_t\right|^qdxdt\\[2mm]
&&\leq\delta^{-\frac{q}{p}}\int^T_{S}E^\beta(t)\phi'(t)|\Omega_t|^{1-\frac{q}{2}}\left(\int_{\Omega_t}\left|\nabla u_t\right|^2dx\right)^{\frac{q}{2}}dt\\[2mm]
&&=\delta^{-\frac{q}{p}}\int^T_{S}E^\beta(t)\phi'(t)|\Omega_t|^{1-\frac{q}{2}}\left[-E'(t)\right]^{\frac{q}{2}}dt\\[2mm]
&&\leq\delta^{-\frac{q}{p}}\eta\int^T_{S}\left[E^\beta(t)\phi'(t)|\Omega_t|^{1-\frac{q}{2}}\right]^{\frac{1}{1-\frac{q}{2}}}dt+\delta^{-\frac{q}{p}}\eta^{1-\frac{2}{q}}\int^T_{S}\left[-E'(t)\right]dt.
\end{eqnarray*}
Assume
\begin{equation}\label{4.28}
\phi'^{\frac{1}{1-\frac{q}{2}}}(t)|\Omega_t|\leq C\phi'(t),
\end{equation}
Then we conclude that
\begin{eqnarray}\label{4.29}
T_5\!\!\!\!\!\!\!\!&&\leq \delta p\int^T_{S}E^{\beta+1}(t)\phi'(t)dt+C\delta^{-\frac{q}{p}}\eta\int^T_{S}E^{\frac{\beta}{1-\frac{q}{2}}}(t)\phi'(t)dt\nonumber\\[2mm]
&&\quad+\delta^{-\frac{q}{p}}\eta^{1-\frac{2}{q}}\int^T_{S}\left[-E'(t)\right]dt.
\end{eqnarray}
In order to make \eqref{4.17}, \eqref{4.20}, \eqref{4.24}, \eqref{4.25} and \eqref{4.28} true, let
\begin{equation}\label{4.27}
\phi'(t)\sim |\Omega_t|^{-m}, \quad m\geq2.
\end{equation}
At this time, it is easy to check that \eqref{4.17}, \eqref{4.20} and \eqref{4.28} hold. In the following, for $\phi'(t)$ like \eqref{4.27}, we verify that \eqref{4.24} and \eqref{4.25} is also true. Without loss of generality,

\vspace{0.5mm}
let $\phi'(t)=C_1|\Omega_t|^{-m}.$ It is easy to get
$$\phi''(t)=-C_1m|\Omega_t|^{-m-1}\frac{d}{dt}|\Omega_t|=-C_1^{-\frac{1}{m}}m\phi'^{\frac{m+1}{m}}(t)\frac{d}{dt}|\Omega_t|.$$
Furthermore
\begin{eqnarray*}
|\phi''(t)|^{\alpha p}|\Omega_t|^p=C_1^{-\frac{\alpha p}{m}}m^{\alpha p}\phi'^{\left(\frac{m+1}{m}\right)\alpha p}(t)\left|\frac{d}{dt}|\Omega_t|\right|^{\alpha p}|\Omega_t|^p.
\end{eqnarray*}
\eqref{4.24} holds, if $\left|\frac{d}{dt}|\Omega_t|\right|\leq C$ and
\begin{eqnarray*}
\phi'^{\left[\left(\frac{m+1}{m}\right)\alpha p-1\right]}(t)|\Omega_t|^p\leq C.
\end{eqnarray*}
This, by \eqref{4.27}, is equivalent to
\begin{eqnarray*}
|\Omega_t|^{-m\left[\left(\frac{m+1}{m}\right)\alpha p-1\right]+p}\leq C.
\end{eqnarray*}
If $|\Omega_t|$ is infinitely increasing, then one has $-m\left[\left(\frac{m+1}{m}\right)\alpha p-1\right]+p\leq 0$. It follows that
\begin{eqnarray*}
m\geq\frac{(1-\alpha)p}{\alpha p-1}, \quad \alpha>\frac{1}{p}.
\end{eqnarray*}
On the other hand,
\begin{eqnarray*}
\left[|\phi''(t)|^{(1-\alpha )q}|\Omega_t|^{1+\frac{q}{2}}\right]^{\frac{1}{1-\frac{q}{2}}}=m^{\frac{(1-\alpha )q}{1-\frac{q}{2}}}\phi'^{\left(\frac{m+1}{m}\right)\frac{(1-\alpha )q}{1-\frac{q}{2}}}(t)\left|\frac{d}{dt}|\Omega_t|\right|^{\frac{(1-\alpha )q}{1-\frac{q}{2}}}|\Omega_t|^{\frac{1+\frac{q}{2}}{1-\frac{q}{2}}}.
\end{eqnarray*}
\eqref{4.25} holds, if $\left|\frac{d}{dt}|\Omega_t|\right|\leq C$ and
\begin{eqnarray*}
\phi'^{\left[\left(\frac{m+1}{m}\right)\frac{(1-\alpha )q}{1-\frac{q}{2}}-1\right]}(t)|\Omega_t|^{\frac{1+\frac{q}{2}}{1-\frac{q}{2}}}\leq C.
\end{eqnarray*}
This, by \eqref{4.27}, is equivalent to
\begin{eqnarray*}
|\Omega_t|^{-m\left[\left(\frac{m+1}{m}\right)\frac{(1-\alpha )q}{1-\frac{q}{2}}-1\right]+\frac{1+\frac{q}{2}}{1-\frac{q}{2}}}\leq C.
\end{eqnarray*}
If $|\Omega_t|$ is infinitely increasing, then one has $-m\left[\left(\frac{m+1}{m}\right)\frac{(1-\alpha )q}{1-\frac{q}{2}}-1\right]+\frac{1+\frac{q}{2}}{1-\frac{q}{2}}\leq 0,$ $q=\frac{p}{p-1}.$ It follows that
\begin{eqnarray*}
m\geq\frac{\left(\frac{1}{2}+\alpha\right)p-1}{\left(\frac{1}{2}-\alpha\right)p+1}, \quad \alpha<\frac{1}{2}+\frac{1}{p}.
\end{eqnarray*}
To sum up, when
\begin{equation}\label{4.23}
m\geq \max\left\{2, \frac{(1-\alpha)p}{\alpha p-1}, \frac{\left(\frac{1}{2}+\alpha\right)p-1}{\left(\frac{1}{2}-\alpha\right)p+1}\right\}, \quad \frac{1}{p}<\alpha<\frac{1}{2}+\frac{1}{p},
\end{equation}
$\phi'(t)$ like\eqref{4.27} satisfies \eqref{4.17}, \eqref{4.20}, \eqref{4.24}, \eqref{4.25} and \eqref{4.28}.

\vspace{0.5mm}
Let $\frac{\beta}{1-\frac{q}{2}}=\beta+1$ and $\varepsilon,$ $\eta,$ $\delta$ be sufficiently small in \eqref{4.26} and \eqref{4.29}, combining \eqref{4.16}, \eqref{4.18}, \eqref{4.21} and \eqref{4.22}, we end up with
\begin{eqnarray*}
\int^T_{S}E^{\beta+1}(t)\phi'(t)dt\leq C\left[E^{\beta+1}(S)+E(S)\right]\leq C\left[E^\beta(0)+1\right]E(S).
\end{eqnarray*}
Let $T\rightarrow+\infty$, we get
\begin{eqnarray*}
\int^{+\infty}_{S}E^{\beta+1}(t)\phi'(t)dt\leq C\left[E^{\beta+1}(S)+E(S)\right]\leq \frac{C\left[E^\beta(0)+1\right]}{E^\beta(0)}E^\beta(0)E(S).
\end{eqnarray*}
By Lemma \ref{lem4.2} and choosing $\phi(t)=(1+kt)^\gamma-1,\ k>0,\ 0<\gamma<1,$ one obtains
\begin{eqnarray*}
E(t)\leq E(0)\left(\frac{1+\beta}{1+\beta A\phi(t)}\right)^{\frac{1}{\beta}}\leq E(0)\left(\frac{1+\beta}{\beta A}\right)^{\frac{1}{\beta}}\phi(t)^{-\frac{1}{\beta}},
\end{eqnarray*}
where $A=\frac{E^\beta(0)}{C\left(E^\beta(0)+1\right)}.$

\vspace{0.5mm}
More precisely,
\begin{equation*}
E(t)\leq \left[\frac{C(1+\beta)\left(E^\beta(0)+1\right)}{\beta}\right]^{\frac{1}{\beta}}\left[(1+kt)^\gamma-1\right]^{-\frac{1}{\beta}}, \quad \beta=\frac{p-2}{p}, \quad p>2.
\end{equation*}
At this time, $\phi'(t)=k\gamma(1+kt)^{\gamma-1}$. Letting $\phi'(t)=k\gamma|\Omega_t|^{-m},$ one has
\begin{eqnarray*}
|\Omega_t|=(1+kt)^{\frac{1-\gamma}{m}},
\end{eqnarray*}
where $m$ is given in \eqref{4.23}.

\vspace{0.5mm}
If $p=2$, then let $\beta=0$. In this case, it is required that $\phi'(t)$ satisfies
\begin{equation}\label{4.30}
\phi'(t)|\Omega_t|^2\leq C,
\end{equation}
\begin{equation}\label{4.31}
\phi'^2(t)|\Omega_t|^2\leq C,
\end{equation}
\begin{equation}\label{4.32}
|\phi''(t)|^{2\alpha}|\Omega_t|^2\leq C\phi'(t),
\end{equation}
\begin{equation}\label{4.33}
|\phi''(t)|^{2(1-\alpha)}|\Omega_t|^2\leq C,
\end{equation}
where \eqref{4.33} is a little different compared to the former($p>2$), but it is more simpler.

\vspace{0.5mm}
Let $\phi'(t)\sim |\Omega_t|^{-m},$ $m\geq2$. It is easy to see that $\phi'(t)$ satisfies \eqref{4.30} and \eqref{4.31}.
If $\left|\frac{d}{dt}|\Omega_t|\right|\leq C$, then $\phi''(t)\sim \phi'^{\frac{m+1}{m}}(t)$.

\vspace{0.5mm}
\eqref{4.32} holds, if
\begin{equation*}
\phi'(t)^{\left[2\alpha\left(\frac{m+1}{m}\right)-1\right]}|\Omega_t|^2\leq C.
\end{equation*}
Further
$$|\Omega_t|^{-m\left[2\alpha\left(\frac{m+1}{m}\right)-1\right]+2}\leq C.$$
If $|\Omega_t|$ is infinitely increasing, then one has $-m\left[2\alpha\left(\frac{m+1}{m}\right)-1\right]+2\leq0$. It follows that
\begin{equation*}
m\geq\frac{2(1-\alpha)}{2\alpha-1}, \quad \alpha>\frac{1}{2}.
\end{equation*}
\eqref{4.33} holds, if
\begin{equation*}
\phi'(t)^{2(1-\alpha)\left(\frac{m+1}{m}\right)}|\Omega_t|^2\leq C.
\end{equation*}
Similarly, we get
\begin{equation*}
m\geq\frac{\alpha}{1-\alpha}.
\end{equation*}
Let $\alpha=\frac{2}{3}$. We deduce for $m\geq2$, it holds that
\begin{equation*}
\int^{+\infty}_{S}E(t)\phi'(t)dt\leq CE(S).
\end{equation*}
Also choosing $\phi(t)=(1+kt)^\gamma-1$, $k>0$, $0<\gamma<1$, and
by Lemma\ref{lem4.2}, we obtain
\begin{equation*}
E(t)\leq E(0)e^{1-\frac{1}{C}\left[(1+kt)^\gamma-1\right]}.
\end{equation*}
At this time,
$\phi'(t)=k\gamma(1+kt)^{\gamma-1}$. Letting $\phi'(t)=k\gamma|\Omega_t|^{-m}$, we see that
\begin{equation*}
|\Omega_t|=(1+kt)^{\frac{1-\gamma}{m}}, \quad   m\geq2.
\end{equation*}
\end{proof}

\begin{remark}
If $|\Omega_t|$ is bounded, then

\medskip
\noindent$(1)$ for $p>2$, choosing $\phi(t)=t$, one has
\begin{equation*}
E(t)\leq E(0)\left(\frac{1+\beta}{1+\beta At}\right)^{\frac{1}{\beta}}, \quad \beta=\frac{p-2}{p};
\end{equation*}

\noindent$(2)$ for $p=2$, putting $\beta=0$, and also choosing $\phi(t)=t$, one gets
\begin{equation*}
E(t)\leq E(0)e^{1-\frac{1}{C}t}.
\end{equation*}
\end{remark}

\section{Further works}
In this article, we get a decay estimate $E(t)\leq E(0)e^{1-\frac{1}{C}\left[(1+kt)^\gamma-1\right]}$ for the energy of system \eqref{4.1} in the case of domains with $|\Omega_t|=(1+kt)^{\frac{1-\gamma}{m}},$ $k>0,$ $0<\gamma<1,$ $m\geq2.$ However, the method we adopt here can not be applied to the case of domains with $|\Omega_t|=(1+kt)^{\alpha},$ $\alpha>\frac{1}{2}.$ In a forthcoming paper, we are going to study the stability of p-Laplacian wave equations in a more general domain.


\begin{thebibliography}{1}{\small}




\bibitem{BM}  Benaissa A, Mokeddem S, {\it
Decay estimates for the wave equation of p-Laplacian type with dissipation of m-Laplacian type},
Math. Methods Appl. Sci.  30 (2007), no. 2, 237--247.








\bibitem{FRS}  Ferreira J, Raposo C A, Santos M L, {\it Global existence for a quasilinear hyperbolic equation in a noncylindrical domain},
 Int. J. Pure Appl. Math. 29 (2006), no. 4, 457--467.





\end{thebibliography}
\end{document}